\RequirePackage[l2tabu, orthodox]{nag}
\RequirePackage{fixltx2e}
\documentclass[10pt,a4paper,reqno,oneside,final]{smfart}

\usepackage{leftidx}
\usepackage[titletoc]{appendix}
\usepackage{tikz}\usetikzlibrary{decorations.markings,matrix,arrows}
\usepackage{tikz-cd}
\usepackage{amsfonts}
\usepackage{amsthm}
\usepackage[T1]{fontenc}
\usepackage[mathscr]{eucal}
\usepackage{setspace}
\usepackage{amssymb,latexsym,amsmath,amscd}
\usepackage{graphicx}
\usepackage{showkeys}
\usepackage[french]{babel}
\usepackage{layout}
\usepackage{enumerate}
\usepackage{indentfirst}
\usepackage[varg]{pxfonts}
\usepackage[stretch=10]{microtype}
\usepackage{booktabs}
\usepackage{xspace}
\usepackage{eufrak}
\usepackage{calrsfs}
\usepackage{filecontents}
\usepackage{mathtools}
\usepackage{titlesec}
\usepackage{fixltx2e}
\usepackage{todonotes}
\usepackage[backref = page, colorlinks   = true,
     citecolor    = gray]{hyperref}

\titleformat{\paragraph}
{\normalfont\normalsize\bfseries}{\theparagraph}{1em}{}
\titlespacing*{\paragraph}
{0pt}{3.25ex plus 1ex minus .2ex}{1.5ex plus .2ex}

\makeatletter
\newcommand{\neutralize}[1]{\expandafter\let\csname c@#1\endcsname\count@}
\makeatother

\theoremstyle{plain}
\newtheorem{thm}{Theorem}[section]

\newtheorem*{thm*}{Theorem}

\newtheorem{lem}[thm]{Lemma}
\newtheorem*{claim}{Claim}
\newtheorem{pro}[thm]{Proposition}
\newtheorem{pro-def}[thm]{Proposition-Definition}
\newtheorem{cor}[thm]{Corollary}
\newtheorem{conj}[thm]{Conjecture}
\newtheorem{que}[thm]{Question}
\newtheorem{prob}[thm]{Problem}
\newtheorem{Def}[thm]{Definition}

\theoremstyle{definition}

\newtheorem{rem}[thm]{Remark}

\theoremstyle{remark}

\newcommand{\ssec}{\subsection}

\newcommand{\wt}{\widetilde}

\newcommand{\bP}{\mathbf{P}}

\newcommand{\bN}{\mathbf{N}}
\newcommand{\bC}{\mathbf{C}}

\newcommand{\bR}{\mathbf{R}}

\newcommand{\bQ}{\mathbf{Q}}

\newcommand{\bZ}{\mathbf{Z}}

\newcommand{\gO}{\Omega}
\newcommand{\ga}{\alpha}

\newcommand{\gk}{\kappa}
\newcommand{\gT}{\Theta}

\newcommand{\go}{\omega}
\newcommand{\gs}{\sigma}
\newcommand{\gD}{\Delta}

\newcommand{\cH}{\mathcal{H}}
\newcommand{\cD}{\mathcal{D}}
\newcommand{\cC}{\mathcal{C}}
\newcommand{\cU}{\mathcal{U}}

\newcommand{\cE}{\mathcal{E}}
\newcommand{\cJ}{\mathcal{J}}
\newcommand{\cV}{\mathcal{V}}

\newcommand{\cS}{\mathcal{S}}
\newcommand{\cZ}{\mathcal{Z}}

\newcommand{\cX}{\mathcal{X}}
\newcommand{\cY}{\mathcal{Y}}
\newcommand{\cO}{\mathcal{O}}

\newcommand{\ep}{\varepsilon}

\newcommand{\bH}{\mathbf{H}}

\newcommand{\fU}{\mathfrak{U}}

\newcommand{\colonec}{\mathrel{:=}}
\newcommand{\bss}{\backslash}

\newcommand{\const}{\mathrm{const}}

\newcommand{\Gal}{\mathrm{Gal}}

\newcommand{\Id}{\mathrm{Id}}

\newcommand{\pr}{\mathrm{pr}}

\newcommand{\Aut}{\mathrm{Aut}}

\newcommand{\mmin}{\mathrm{min}}

\newcommand{\Map}{\mathrm{Map}}

\newcommand{\Sing}[1]{\mathrm{Sing}\left( #1 \right)}

\newcommand{\dr}{\partial}

\newcommand{\cupp}{\mathbin{\smile}}
\renewcommand{\(}{\left(}
\renewcommand{\)}{\right)}

\newcommand{\hto}{\hookrightarrow}
\newcommand{\dto}{\dashrightarrow}
\newcommand{\<}{\mathopen{<}}
\renewcommand{\>}{\mathclose{>}}
\newcommand{\vast}{\bBigg@{4}}
\newcommand{\Vast}{\bBigg@{5}}
\newcommand{\xto}[1]{\xrightarrow{ #1 }}

\makeatletter
\let\orgdescriptionlabel\descriptionlabel
\renewcommand*{\descriptionlabel}[1]{%
  \let\orglabel\label
  \let\label\@gobble
  \phantomsection
  \edef\@currentlabel{#1}%
  \let\label\orglabel
  \orgdescriptionlabel{#1}%
}
\makeatother
\tikzset{node distance=2cm, auto}

\numberwithin{equation}{section}

\title{Algebraic approximations of compact Kähler threefolds of Kodaira dimension 0 or 1} 

\author{Hsueh-Yung Lin}

\address{
Mathmatisches Institut \\
Universität Bonn \\
Endenicher Allee 60, Office 303 \\
53115 Bonn, Germany}
 \email{linhsueh@math.uni-bonn.de}

\begin{document}

\begin{abstract}
We prove that every compact Kähler threefold $X$ of Kodaira dimension $\kappa = 0$ or $1$ has a $\bQ$-factorial bimeromorphic model $X'$ with at worst terminal singularities such that for each curve $C \subset X'$, the pair $(X',C)$ admits a locally trivial algebraic approximation such that the restriction of the deformation of $X'$ to some neighborhood of $C$ is a trivial deformation. As an application, we prove that every compact Kähler threefold with $\kappa = 0$ or $1$ has an algebraic approximation. We also point out that in order to prove the existence of algebraic approximations of a compact Kähler threefold with $\kappa = 2$, it suffices to prove that of an elliptic fibration over a surface.
\end{abstract}

\maketitle

\section{Introduction}

From the point of view of the Hodge theory, compact Kähler manifolds can be considered as a natural generalization of smooth complex projective varieties. While an arbitrarily small deformation as a complex variety of a smooth complex projective variety might no longer be projective, a sufficiently small deformation of a Kähler manifold remains Kähler. The so-called Kodaira problem asks wether it is possible to obtain all compact Kähler manifolds through (arbitrarily small) deformations of projective varieties.

\begin{prob}[Kodaira problem]\label{prob-Kod}
Given a compact Kähler manifold $X$, does $X$ always admit an (arbitrarily small) deformation to some projective variety?
\end{prob}

In dimension 1, compact complex curves are already projective. For surfaces, Problem~\ref{prob-Kod} is known to have a positive answer, first due to Kodaira using the classification of compact complex surfaces~\cite{KodairaSurfaceII}, then to N. Buchdahl~\cite{Buch2} proving that any compact Kähler surface has an algebraic approximation using M. Green's density criterion (\emph{cf.} Theorem~\ref{thm-Grenndensecrit}). We refer to~\cite{CaoJApproxalg, GrafDefKod0, HYLbimkod1, ClaudonToridefequiv, Schrackdefo} for other positive results.

As for negative answers, C. Voisin constructed  in each dimension $\ge 4$ examples of compact Kähler manifolds which do not have the homotopy type of a smooth projective variety~\cite{Voisincs}, thus answered in particular negatively the Kodaira problem. Later on, she constructed in each even dimension $\ge 8$ examples of compact Kähler manifolds all of whose smooth bimeromorphic models are homotopically obstructed to being a projective variety~\cite{VoisinBiratKod}.

For threefolds, the Kodaira problem remains open at present. There are nevertheless positive results concerning a bimeromorphic variant of the Kodaira problem.

\begin{thm}[$\gk = 0$:~\cite{GrafDefKod0}, $\gk = 1$:~\cite{HYLbimkod1}]\label{thm-bimekod}
Let $X$ be a compact Kähler threefold of Kodaira dimension $\gk = 0$ or $1$. There exists a $\bQ$-factorial bimeromorphic model $X'$ of $X$ with at worst terminal singularities such that $X'$ has a locally trivial algebraic approximation.
\end{thm}

In order to prove Theorem~\ref{thm-bimekod}, thanks to the minimal model program (MMP) for Kähler threefolds~\cite{HorPet}, we can choose  $X'$ to be a minimal model of $X$, and this is what we did in most of the cases. Geometric descriptions of these varieties $X'$ can be obtained as an output of the abundance conjecture~\cite{CHPabun} applied to $X'$, which is enough to prove the existence of a locally trivial algebraic approximation for $X'$. 

The aim of this article is to prove the following stronger version of Theorem~\ref{thm-bimekod} by further exploiting the geometry of $X'$. We refer to Section~\ref{ssec-term} for the terminologies used in the statement of Theorem~\ref{thm-mainpair}.

\begin{thm}\label{thm-mainpair}
Let $X$ be a compact Kähler threefold of Kodaira dimension $\gk = 0$ or $1$. There exists a $\bQ$-factorial bimeromorphic model $X'$ with at worst terminal singularities such that whenever $C \subset X'$ is a curve or empty, the pair $(X',C)$ has a locally trivial and $C$-locally trivial algebraic approximation.
\end{thm}

We will also prove a result relating the type of algebraic approximation that $X'$ has in Theorem~\ref{thm-mainpair} and the algebraic approximation of $X$. 

\begin{pro}\label{pro-red}
Let $X$ be a compact Kähler threefold and $X'$ a normal bimeromorphic model of $X$. If $(X',C)$ has a locally trivial and $C$-locally trivial algebraic approximation whenever $C \subset X'$ is a curve or empty, then $X$ has an algebraic approximation.
\end{pro}

We refer to Corollary~\ref{cor-red} for a more general statement. Since $\bQ$-factorial varieties are normal by definition, putting Proposition~\ref{pro-red} together with Theorem~\ref{thm-mainpair} yields immediately the following result.

\begin{thm}\label{thm-main}
Every compact Kähler threefold of Kodaira dimension 0 or 1 has an algebraic approximation.
\end{thm}

As for threefolds of Kodaira dimension 2, since minimal models of such varieties are elliptic fibrations, the existence of algebraic approximations of these varieties is related to the following question.

\begin{que}\label{Q-ellpaa}
Let $f : Y \to B$ be an elliptic fibration where $Y$ is a compact Kähler and the base $B$ is smooth and projective. Assume that the locus $D \subset B$ parameterizing singular fibers of $f$ is normal crossing, does $Y$ have an algebraic approximation?
\end{que}

We will see that a positive solution of Question~\ref{Q-ellpaa} will eventually solve the Kodaira problem for threefolds of Kodaira dimension 2. 

\begin{pro}\label{pro-gk2aa}
If Question~\ref{Q-ellpaa} has a positive answer in the case where $B$ is a surface, then every compact Kähler threefold of Kodaira dimension 2 has an algebraic approximation.
\end{pro}

In view of~\cite[Theorem 1.1]{ClaudonToridefequiv} and~\cite[Theorem 1.6]{HYLbimkod1}, it is plausible that Question~\ref{Q-ellpaa} would have a positive answer. It is a work in progress of Claudon and Höring toward an answer to Question~\ref{Q-ellpaa}.

The article is organized as follows. We will first introduce in Section~\ref{sec-deformation} some deformation-theoretic terminologies including those appearing in Theorem~\ref{thm-mainpair} then prove some general results. In particular, we will prove Corollary~\ref{cor-red} and deduce Proposition~\ref{pro-red} from it. Next, we will turn to describing  minimal models of a compact Kähler threefold of Kodaira dimension $0$ or $1$ in Section~\ref{sec-bim}. According to these descriptions, we will choose some threefolds $X$ and prove in Section~\ref{sec-casparcas} that whenever $C \subset X$ is a curve or empty, the pair $(X,C)$ always has a $C$-locally trivial algebraic approximation. Based on these results, the proof of Theorem~\ref{thm-mainpair} will be concluded in Section~\ref{sec-concl}, where we also prove Proposition~\ref{pro-gk2aa}.

\section{Deformations}\label{sec-deformation}

\ssec{Terminologies}\label{ssec-term}
 \hfill
 
Let $X$ be a complex variety. A \emph{deformation} of $X$ is a surjective flat holomorphic map $\pi: \cX \to \gD$ containing $X$ as a fiber. We say that a deformation $\pi: \cX \to \gD$ is \emph{locally trivial} if for every $x \in \cX$, there exists a neighborhood $x \in \cU \subset \cX$ of $x$ such that if $U \colonec \pi^{-1}(\pi(x)) \cap \cU$, then $\cU$ is isomorphic to $U \times \pi(\cU)$ over $\pi(\cU)$. 

In this article, a \emph{fibration} is a surjective holomorphic map $f: X \to B$ with connected fibers. A deformation $\pi : \cX \to \gD$ of $X$ is called \emph{strongly locally trivial} with respect to the fibration structure $f : X \to B$ if $\pi$ has a factorization of the form
$$
\begin{tikzcd}[cramped, row sep = 20, column sep = 40]
\cX  \arrow[r, "q"] \arrow[d, "\pi", swap] & \gD \times B  \ar{dl}{\pr_1} \\
\gD &  \\
\end{tikzcd}
$$
such that the restriction of $q$ to $X$ onto its image coincides with $f$, and that for every $(t,b) \in \gD \times B$, there exist neighborhoods $b \in U \subset B$ and $t \in V \subset \gD$ such that $q^{-1}(V \times U)$ is isomorphic to $q^{-1}(\{t\} \times U) \times V$ over $V$.

Let $X$ be a complex variety and $C \subset X$ a subvariety of $X$. A \emph{$C$-locally trivial deformation of $(X,C)$} is a deformation $(\cX,\cC) \to \gD$ of the pair $(X,C)$ such that the deformation $(\cU,\cC) \to \gD$ restricted to some neighborhood $\cU \subset \cX$ of $\cC$  is isomorphic to the trivial deformation $\(U \times \gD, C \times \gD\) \to \gD$ with $U \colonec \cU \cap X$. An \emph{algebraic approximation} of the pair $(X,C)$ is a deformation $(\cX,\cC) \to \gD$ of $(X,C)$ such that there exists a sequence of points $(t_i)_{i\in \bN}$ in $\gD$ parameterizing algebraic members and converging to $o$, the point which parameterizes $(X,C)$.

 If $X$ is endowed with a $G$-action where $G$ is a group and $C$ is a $G$-invariant subvariety, then a \emph{$G$-equivariant deformation of the pair $(X,C)$} is a deformation $(\cX,\cC) \to \gD$ of $(X,C)$ such that the $G$-action on $X$ extends to an action on $\cX$ preserving each fiber of $\cX \to \gD$ and $\cC$.

\ssec{Locally trivial deformations and bimeromorphic transformations}

\hfill

The following lemma concerns the behaviour of $C$-locally trivial deformations of a pair $(X,C)$ under bimeromorphic transformations.

\begin{lem}\label{lem-defcontrpoint}
Let $f : X \to Y$ be a map between complex varieties and assume that there exists a subvariety $C \subset Y$ such that $f$ maps $X \bss D$ isomorphically onto $Y \bss C$ where $D \colonec f^{-1}(C)$.  Then for every $C$-locally trivial deformation $\pi : (\cY,\cC) \to \Delta$ of $Y$, there exists a $D$-locally trivial deformation $(\cX,\cD) \to \gD$ of the pair $\(X, D\)$ together with a map $F: \cX \to \cY$ over $\gD$ such that $F^{-1}(\cC) = \cD$ and that $F_{ | \cX \bss \cD}$ is an isomorphism onto $\cY \bss \cC$.
\end{lem}
\begin{proof}
Let $\cU \subset \cY$ be a neighborhood of $\cC$ such that there exists an isomorphism over $\Delta$ of the pairs  
\begin{equation}\label{isom-neighpinYipair}
\(\cU , \cC\) \simeq \(U \times \Delta, C \times \Delta \)
\end{equation}
where $U \colonec Y \cap \cU$. So we can write
$$\cY \simeq \( (\cY \bss \cC) \sqcup \(U \times \gD \)\) {/} \sim$$
where $\sim$ glues the two pieces of the union using isomorphism~\eqref{isom-neighpinYipair}. Isomorphism~\eqref{isom-neighpinYipair} also implies that since $f$ maps $X \bss D$ isomorphically onto $Y \bss C$, we have over $\gD$
\begin{equation}\label{isom-neighpinYi}
\cU \bss \cC \simeq f^{-1}(U \bss C) \times \gD.
\end{equation}
We define 
$$\cX \colonec \( (\cY \bss \cC) \sqcup (f^{-1}(U)  \times \gD )\) \big{/} \sim $$
and
$$\cD \colonec D \times \gD \subset \cX$$
where $\sim$ glues the two pieces of the union using isomorphism~\eqref{isom-neighpinYi}. One easily checks that $\cX$ is Hausdorff so that $\cX$ is a complex variety. The map $\pi' : \cY  \bss \cC \to \Delta$ and the projection $\pi'' : f^{-1}(U)  \times \gD \to \gD$ give rise to a map 
$$\pi_X : (\cX, \cD) \to \Delta$$ 
which, by construction, is a $D$-locally trivial deformation of the pair $(X,D)$.  Finally the restriction of $f$ to $f^{-1}(U)$  defines an obvious map $ F: \cX \to \cY$ satisfying the property that $F^{-1}(\cC) = \cD$ and that $F_{ | \cX \bss \cD} : \cX \bss \cD \simeq \cY \bss \cC$.
\end{proof}

\begin{rem}
We can also show that given a $D$-locally trivial deformation $(\cX,\cD) \to \gD$ of the pair $\(X, D\)$, there exists a $C$-locally trivial deformation $(\cX,\cC) \to \gD$ of the pair $\(X, C\)$ together with a map $F: \cX \to \cY$ over $\gD$ such that $F^{-1}(\cC) = \cD$ and that $F_{ | \cX \bss \cD}$ is an isomorphism onto $\cY \bss \cC$. This can be proven by exchanging the role of $C$ and $D$ in the proof of Lemma~\ref{lem-defcontrpoint}.
\end{rem}

Let $X$ be a compact Kähler manifold. Assume that $X$ is bimeromorphic to a compact Kähler variety $Y$. After a sequence of blow-ups of $X$ along smooth centers, we obtain a resolution
\begin{equation}\label{map-res}
\begin{tikzcd}[cramped, row sep = 0, column sep = 20]
X & Z  \arrow[r, "\nu"] \ar[l, "\eta" , swap] & {Y} \\
\end{tikzcd}
\end{equation}
of the bimeromorphic map $X \dto Y$. Let $C \subset Y$ be the image of the exceptional set of $\nu$. The following lemma shows in particular that a $C$-locally trivial deformation of the pair $(Y,C)$ always induces a deformation of $X$.

\begin{lem}\label{lem-deform}
Suppose that  $\pi : (\cY,\cC) \to \Delta$ is a $C$-locally trivial deformation of the pair $(Y,C)$. Then up to shrinking $\gD$, the deformation $\pi$ induces a deformation
$$
\begin{tikzcd}[cramped, row sep = 0, column sep = 20]
\cX & \cZ  \arrow[r] \ar[l] & {\cY} \\
\end{tikzcd}
$$
of~\eqref{map-res}. 
\end{lem}

\begin{proof}

Since $\nu$ maps $\nu^{-1}(Y \bss C)$ isomorphically onto $Y \bss C$ and since $(\cY,\cC) \to \gD$ is a $C$-locally trivial deformation of the pair $(Y,C)$, by Lemma~\ref{lem-defcontrpoint} there exists a deformation $\cZ \to \gD$ of $Z$ and a map $F: \cZ \to \cY$ over $\gD$ whose restriction to the central fiber is $\nu : Z \to Y$. 

As $\eta_*\cO_{Z} \simeq \cO_X$ and $R^1\eta_* \cO_{Z} = 0$ since $\eta$ is a composition of blow-ups along smooth centers, by~\cite[Theorem 2.1]{RanStabMap} the deformation $\cZ \to \gD$ of $Z$ induces a deformation $\cZ \to \cX$ of the morphism $Z \to X$ over $\gD$ up to shrinking $\gD$.
\end{proof}

The following is an immediate consequence of Lemma~\ref{lem-deform}.

\begin{cor}\label{cor-red}
With the same notation as above, if $Y$ has a $C$-locally trivial algebraic approximation, then $X$ also has an algebraic approximation. In particular, if $Y$ is normal and satisfies the property that for every subvariety $C \subset Y$ whose irreducible components are all of codimension $\ge 2$, the pair $(Y,C)$ has a $C$-locally trivial algebraic approximation, then $X$ also has an algebraic approximation.
\end{cor}

\begin{proof}
Let $\cY \to \gD$ be a $C$-locally trivial algebraic approximation of $Y$ and let
$$
\begin{tikzcd}[cramped, row sep = 0, column sep = 20]
\cX & \cZ  \arrow[r] \ar[l] & {\cY} \\
\end{tikzcd}
$$
be the induced deformation of~\eqref{map-res} as in Lemma~\ref{lem-deform}. Up to shrinking $\gD$ we can suppose that for each $t \in \gD$, the fibers $\cZ_t \to \cX_t$ and $\cZ_t \to \cY_t$ of the maps $\cZ \to \cX$ and $\cZ \to \cY$ over $t$  are both bimeromorphic. Therefore if over a point $t \in \gD$ the variety $\cY_t$ is algebraic, then $\cX_t$ is also algebraic.

For the last statement of Corollary~\ref{cor-red}, the normality of $Y$ implies that each irreducible component of the image in $Y$ of the exceptional set $E$ of $\nu$ is of codimension $\ge 2$. Thus $(Y,\nu(E))$ has a $\nu(E)$-locally trivial algebraic approximation by assumption. We conclude by the first part of Corollary~\ref{cor-red} that $X$ has an algebraic approximation. 
\end{proof}

\begin{proof}[Proof of Proposition~\ref{pro-red}]
Assume that $X'$ satisfies the hypothesis made in the proposition. Let $C \colonec (C_0 \sqcup C_1) \subset X'$ be a subvariety of dimension $\le 1$ where $C_i$ denotes the union of the irreducible components of $C$ of dimension $i$. Since $\dim C_0 = 0$, a locally trivial deformation of $X'$ induces in particular a $C_0$-locally trivial deformation of $(X',C_0)$. Hence by assumption, the pair $(X',C)$ has a $C$-locally trivial  algebraic approximation. It follows from the second part of Corollary~\ref{cor-red} that $X$ has an algebraic approximation.
\end{proof}

\ssec{$G$-equivariant locally trivial deformations}

\hfill

The following lemma shows that given a $G$-equivariant $C$-locally trivial deformation $(\cX,\cC) \to \gD$ of $(X,C)$, there always exists a \emph{$G$-equivariant} trivialization of some neighborhood of $\cC$. This will imply that the quotient $(\cX/G,\cC/G) \to \gD$ is a $C/G$-locally trivial deformation of $(X/G,C/G)$.

\begin{lem}\label{lem-Gloctriv}
Let $X$ be a smooth complex variety and $G$ a finite group acting on $X$. Let $C$ be a $G$-invariant subvariety of $X$ and assume that there exists a $G$-equivariant deformation of $\pi : \cX \to \gD$ of $X$ over a one-dimensional base $\gD$.  Assume also that there exists an open subset $\cV \subset \cX$ and an isomorphism $\cV \simeq V \times \gD$ over $\gD$ where $V \colonec \cV  \cap X$ such that $V$ contains $C$ (this hypothesis holds for instance, when $\pi$ induces a $G$-equivariant $C$-locally trivial deformation of $(X,C)$), then up to shrinking $\gD$, there exist $\cC \subset \cX$, a $G$-invariant neighborhood $\cU$ of $\cC$, and a $G$-equivariant isomorphism
$$(\cU,\cC) \simeq (U \times \gD, C \times \gD)$$
over $\gD$ where $U \colonec \cU \cap X$. 

In particular, $\pi: (\cX, \cC) \to \gD$ is a $G$-equivariant $C$-locally trivial deformation of $(X,C)$ and the quotient $(\cX/G,\cC/G) \to \gD$ is a locally trivial and $C/G$-locally trivial deformation of $(X/G,C/G)$.
\end{lem}

Before proving Lemma~\ref{lem-Gloctriv}, let us first prove a technical lemma.

\begin{lem}\label{lem-techG}
Let $G$ be a finite group acting on a variety $X$ and let $\pi: \cX \to \gD$ be a $G$-equivariant deformation of $X$ over a one-dimensional base. Let $\cV \subset \cX$ be an open subset such that there exists an isomorphism $\cV \simeq V \times \gD$ over $\gD$ where $V \colonec \cV \cap X$. Let 
$$\cV^G \colonec \bigcap_{g \in G} g(\cV).$$ 
Then for every $G$-invariant relatively compact subset $U \subset V^G \colonec \cV^G \cap X$, up to shrinking $\gD$ there exists a $G$-invariant subset $\cU$ of $\cV^G$ and a \emph{$G$-equivariant} isomorphism $\cU \simeq U \times \gD$ over $\gD$.
\end{lem}

\begin{proof}
We may assume that $V^G \ne \emptyset$.
Since $\cV^G$ is open by finiteness of $G$, after shrinking $\gD$ we can also assume that the restriction of $\pi$ to $\cV^G$ is surjective and that $\gD$ is isomorphic to the open unit disc $B(0,1) \subset \bC$ such that $0$ parameterizes the central fiber $X$. Fix a generator $\frac{\dr}{\dr t}$ of the space of constant vector fields $\Gamma(\gD,T_{\gD})_{\const} \simeq \bC$ on $\gD$. For $z \in \bC$, let $z \frac{\dr}{\dr t} \in \Gamma(\gD,T_{\gD})_{\const}$ denote the corresponding vector field. 

By identifying $\cV^G$ with a subset of $V \times \gD $ through the isomorphism $\cV \simeq V \times \gD$, we can define the homomorphism of Lie algebras
\begin{equation}
\begin{split}
\xi : \bC & \to  \Gamma(\cV^G , T_{\cV^G}) \\
z & \mapsto  \sum_{g \in G} g^* \(\chi(z)  _{| \cV^G}\), 
\end{split}
\end{equation}
where $\chi (z )$ is the vector field on $V \times \gD$ which projects to $z  \frac{\dr}{\dr t}$ in $\gD$ and to $0$ in $V$.  
By~\cite[Satz 3]{Kaup} (see also~\cite[Theorem 5.3]{GrafDefKod0}), there exists a local group action 
$$\Phi : \gT \to \cV^G$$ 
of $\bC$ on $\cV^G$ inducing $\xi$, where $\gT \subset \bC \times \cV^G$ is a neighborhood of $\{0\} \times \cV^G$. We recall that the meaning of a local group action is the following.
\begin{enumerate}[i)]
\item For all $x \in \cV^G$, the subset $\gT \cap \(\bC  \times \{x\}\)$ is connected.
\item $\Phi(0,\bullet)$ is the identity map on $\cV^G$.
\item $\Phi(gh,x) = \Phi(g,\Phi(h,x))$ whenever it is well-defined.
\item The morphism of Lie algebras $ \bC  \to  \Gamma(\cV^G , T_{\cV^G})$ induced by $\Phi$ coincides with $\xi$.
\end{enumerate}
Since  the vector field $\xi(z)$ is $G$-invariant for all $z \in \bC$ by construction, the map $\Phi$ is also $G$-equivariant (where $G$ acts trivially on $\bC$). Also since $G$ acts on $\cV^G \to \gD$ in a fiber-preserving way, the projection of $\xi(z)$ in $\Gamma(\cV^G, \pi^*T_{\gD})$ equals $|G| \cdot p_2^* \(z  \frac{\dr}{\dr t} \)$. Hence if $\Phi_\gD$ denotes the local group action on $\gD$ defined by
\begin{equation*}
\begin{split}
\Phi_\gD : (\Id_{\bC} \times \pi)(\gT)  & \to \gD \\
(x,b) & \mapsto b + |G|\cdot x
\end{split}
\end{equation*}
then we have the following commutative diagram.
 \begin{equation}\label{diag-locGact}
\begin{tikzcd}[cramped, column sep = 20]
\gT \arrow[r, "\Phi"] \ar[d] & \cV^G \ar[d, "\pi"] \\
(\Id_{\bC} \times \pi)(\gT) \ar[r, "\Phi_{\gD}"] & \gD\\
\end{tikzcd}
\end{equation}
By the relative compactness of $U$ inside $V^G$, there exists $\ep > 0$ such that 
$$\fU \colonec B(0,\ep) \times U \subset \gT.$$ 
The restriction of $\Phi$ to $\fU$ is isomorphic onto its image. We verify easily with the help of~\eqref{diag-locGact} and the properties ii) and iii) that the inverse of $\Phi : \fU \to \Phi(\fU)$ is

\begin{equation*}
\begin{split}
\Psi : \Phi(\fU) & \to  \fU\\
v & \mapsto \( \frac{\pi(v)}{|G|}, \Phi\(-\frac{\pi(v)}{|G|}, v\)\).
\end{split}
\end{equation*}

Let $\cU \colonec \Phi\(B\(0,\frac{\ep}{|G|}\) \times U \) \subset \cV^G$. We have $U \colonec \cU \cap X$ by  ii) and up to replacing $\gD$ by $B(0,{\ep})$,  we have thus by construction an isomorphism 
\begin{equation*}
\begin{split}
U \times \gD & \xto{\sim}  \cU \\
(x,t) & \mapsto   \Phi\(\frac{t}{|G|}, x \) , 
\end{split}
\end{equation*}
over $\gD$, which is moreover $G$-equivariant since $\Phi$ is $G$-equivariant.
\end{proof}

\begin{proof}[Proof of Lemma~\ref{lem-Gloctriv}]
Since $C$ is $G$-invariant and since the subset $\cV^G \colonec \bigcap_{g \in G} g(\cV)$ is a finite intersection so is an open subset, $V^G \colonec \cV^G \cap X$ is a $G$-invariant neighborhood of $C$. Let $U \subset V^G$ be a $G$-invariant neighborhood of $Y$ which is relatively compact in $V^G$. By applying Lemma~\ref{lem-techG} to $\cV$ and to $U$, we deduce that up to shrinking $\gD$, there exists a $G$-invariant subset $\cU \subset \cV^G$ together with a $G$-equivariant isomorphism 
$$U \times \gD \simeq \cU$$ 
over $\gD$. As $C$ is a $G$-invariant subset of $U$, the image $\cC \subset \cU$ of $C \times \gD$ under the above isomorphism is also $G$-invariant. This proves that the $G$-equivariant isomorphism $\cU \simeq U \times \gD$ induces a $G$-equivariant isomorphism of the pairs $(\cU,\cC) \simeq (U \times \gD, C \times \gD)$, which is the main statement of the lemma.

It follows by definition that $\pi : (\cX , \cC) \to \gD$ is a $G$-equivariant $C$-locally trivial deformation of $(X,C)$. Since $X$ is smooth, up to further shrinking $\gD$ we can assume that $\cX \to \gD$ is a smooth deformation, so that the quotient $\cX/G \to \gD$ is a locally trivial deformation~\cite[Proposition 8.2]{GrafDefKod0}. As
$$(\cU / G ,\cC / G) \simeq \((U/ G) \times \gD  ,(C/ G) \times \gD\)$$
 over $\gD$, the deformation $(\cX / G ,\cC / G) \to \gD$ of the pair $(X/G,C/G)$ is $C/G$-trivial.
\end{proof}

The  following lemma  is a special case of Lemma~\ref{lem-Gloctriv}.

\begin{lem}\label{lem-strloctrivG}
Let $f : X \to B$ be a $G$-equivariant fibration where $G$ is a finite group. Let 
$$
\begin{tikzcd}[cramped, row sep = 20, column sep = 40]
\cX  \arrow[r, "q"] \arrow[d, "\pi", swap] & \gD \times B  \ar{dl}{\pr_1} \\
\gD &  \\
\end{tikzcd}
$$
be a $G$-equivariant strongly locally trivial deformation of $f$ over a one-dimensional base $\gD$. Suppose that $C$ is a $G$-invariant subvariety of $X$ and that $f(C)$ is a finite set of points, then the deformation $\pi : \cX \to \gD$ induces a $G$-equivariant $C$-locally trivial deformation $(\cX,\cC) \to \gD $ of the pair $(X,C)$.
\end{lem}

\begin{proof}
 Let $\{p_1,\ldots,p_n\} \colonec f(C) \subset B$.  By definition, up to shrinking $\gD$, for each $i$ there exists a neighborhood $p_i \in V_i \subset B$ of $p_i$ such that  the restriction of $\pi : \cX \to \gD$ to $ \cV_i \colonec q^{-1}(\gD \times V_i)$ is isomorphic to $(\cV_i \cap X) \times \gD $ over $\gD$.  Up to shrinking the $V_i$'s, we can assume that they are pairwise disjoint, so that $\cV \colonec \sqcup_{i = 1}^n \cV_i$ is isomorphic to $V \times \gD$ over $\gD$ where $V \colonec \cV \cap X$. Applying Lemma~\ref{lem-Gloctriv} to the $G$-equivariant deformation $\pi:\cX \to \gD$, the $G$-invariant subvariety $C$, and $\cV$ yields Lemma~\ref{lem-strloctrivG}.
\end{proof}

\begin{rem}
For simplicity, Lemma~\ref{lem-techG} is stated and proven under the assumption that $\dim \gD = 1$ and so are Lemma~\ref{lem-Gloctriv} and Lemma~\ref{lem-strloctrivG}, which will be enough for the purpose of this article. All these lemmata could have been stated without assuming that $\dim \gD = 1$.
\end{rem}

\section{Bimeromorphic models of non-algebraic compact Kähler threefolds}\label{sec-bim}

The reader is referred to~\cite{HorPetsurvey} for a survey of the minimal model program (MMP) for Kähler threefolds. Let $X$ be a compact Kähler threefold with non-negative Kodaira dimension $\gk(X)$. By running the MMP on $X$, we obtain a $\bQ$-factorial bimeromorphic model $X_\mmin$ of $X$ with at worst terminal singularities (which are isolated, since $\dim X = 3$) whose canonical line bundle $K_{X_\mmin}$ is nef. Such a variety $X_\mmin$ is called a \emph{minimal model} of $X$. By the abundance conjecture, which is known to be true for Kähler threefolds, there exists $m \in \bZ_{>0}$ such that $mK_{X_\mmin}$ is base-point free and that the surjective map $f : X_\mmin \to B$ defined by the linear system $|mK_{X_\mmin}|$ is a fibration satisfying $\dim B = \gk(B) = \gk(X)$. The fibration $f : X_\mmin \to B$ is called the \emph{canonical fibration} of $X_\mmin$ and a general fiber $F$ of $f$ satisfies $\cO(mK_F) \simeq \cO_F$ by the adjunction formula.

The aim of this section is to describe minimal models of non-algebraic compact Kähler threefolds of Kodaira dimension $\gk= 0$ or $1$. Let us start from varieties with $\gk = 0$.

\begin{pro}\label{pro-gk0MMP} 
Let $X$ be a non-algebraic compact Kähler threefold with $\gk(X) = 0$ and let $X_\mmin$ be a minimal model of $X$. Then $X_\mmin$ is isomorphic to a quotient $\wt{X}/G$ by a finite group $G$ where $\wt{X}$ is either a $3$-torus or a product of a K3 surface and an elliptic curve.
\end{pro}

\begin{proof}
Since $\gk(X) = 0$, there exists $m \in \bZ_{>0}$ such that $\cO(mK_{X_{\mmin}}) \simeq \cO_{X_{\mmin}}$. Let $\pi : \wt{X}_{\mmin} \to X_{\mmin}$ be the index $1$ cover of $X_{\mmin}$: this is a finite cyclic cover étale over $X \bss \Sing{X}$ such that $ K_{\wt{X}_{\mmin}} \simeq \cO_{\wt{X}_{\mmin}}$~\cite[p. 159]{KollarMori}. As ${X}_{\mmin}$ has at worst terminal singularities, by~\cite[Corollary 5.21 (2)]{KollarMori}, the variety $\wt{X}_{\mmin}$ has also at worst terminal singularities. Since $X$ is assumed to be non-algebraic, by~\cite[Theorem 6.1]{GrafDefKod0} $\wt{X}_{\mmin}$ is smooth. Thus by the Beauville-Bogomolov decomposition theorem~\cite[Théorème 1]{Beauvillec1=0}, there exists a finite étale cover ${X}' \to \wt{X}_{\mmin}$ such that ${X}'$ is either a $3$-torus or a product of a K3 surface and an elliptic curve (as $X'$ is non-algebraic, $X'$ cannot be a Calabi-Yau threefold); let $\tau : {X}' \to X_{\mmin}$ denote the composition of ${X}'  \to \wt{X}_{\mmin}$ with $\pi$.

The finite map $\tau$ is étale over $X \bss \Sing{X}$. Let $\wt{X}^\circ \to X' \bss Z \to X \bss \Sing{X}$ be the Galois closure of $\tau_{|X' \bss Z}$ where $Z \colonec \tau^{-1}(\Sing{X})$ and let 
$$G \colonec \Gal\(\wt{X}^\circ/ (X \bss \Sing{X})\).$$ 
Since $\Sing{X}$ and hence $Z$ are finite sets of points, we have $\pi_1(X' \bss Z)  \simeq  \pi_1(X')$. It follows that $\wt{X}^\circ \to X' \bss Z$ extends to $\wt{X} \to X'$ which is the finite étale cover associated to the subgroup $\Gal\(\wt{X}^\circ / (X' \bss Z)\) < \pi_1(X' \bss Z)  \simeq  \pi_1(X')$. The variety $\wt{X}$ is still a $3$-torus or a product of a K3 surface and an elliptic curve. As $\wt{X} \bss \wt{X}^\circ$ is a set of isolated points, the $G$-action on $\wt{X}^\circ$ extends to a $G$-action on $\wt{X}$ whose quotient is $X_\mmin$.
\end{proof}

\begin{rem}
The group $G$ constructed in the proof of Proposition~\ref{pro-gk0MMP} acts freely outside of a finite set of points of $\wt{X}$.
\end{rem}

For  quotients $(S \times E) / G$ of the product of a non-algebraic K3 surface $S$ and an elliptic curve $E$, we can show that the $G$-action is necessarily diagonal.

\begin{lem}\label{lem-GactionK3E}
Let $G$ be a group acting on $S \times E$ where $S$ is a non-algebraic K3 surface and $E$ is an elliptic curve. Then this $G$-action is the product of a $G$-action on $S$ and a $G$-action on $E$.
\end{lem}

\begin{proof}
For each $g \in G$ and each fiber $F$ of the second projection $p_2 : S \times E \to E$, since $h^{0,1}(F) < h^{0,1}(E)$, it follows that $g(F)$ is still a fiber of $p_2$. So the $G$-action on $S \times E$ induces a $G$-action on $E$. Suppose that there exist $g \in G$ and a fiber $E_t$ of the first projection $p_1 : S \times E \to S$ such that $g(E_t)$ is not contracted by $p_1$, then if we vary $t \in S$, we have a two-dimensional covering family of curves 
$$\{E'_t \colonec p_1(g(E_t))\}_{t \in S}$$ 
on $S$ generically of geometric genus $1$. Since  algebraic equivalence coincides with  linear equivalence for curves on a K3 surface and since there is only one-dimensional families of curves of geometric genus $1$ in each linear system,  $\{E'_t\}_{t\in S}$ is in fact a one-dimensional family of curves, say parameterized by some proper curve $T$. As $S$ is non-algebraic, the family $\{E'_t\}_{t\in T}$ is an elliptic fibration and there exists $t \in T$ such that the normalization $\wt{E'_t}$ of $E'_t$ is $\bP^1$. 

Let $C \subset S$ be a curve such that for each $p \in C$, we have $E'_p = E'_t$. Since the curves $g(E_p) \subset S \times E$ are mutually disjoint for $p \in C$, their strict transformations $\wt{g(E_p)}$ in the normalization $\wt{E'_t} \times E$ of $E'_t \times E$ are also disjoint from each other. It follows that $[\wt{g(E_p)}]^2 = 0$ in $H^4(\wt{E'_t} \times E,\bZ)$ and since $\wt{E'_t} \simeq \bP^1$, the curve $\wt{g(E_p)}$ has to be a fiber of $\wt{E'_t} \times E \to \wt{E'_t}$. The latter is in contradiction with the assumption that $g(E_p)$ is not contracted by $p_1$.
\end{proof}

Next we turn to varieties with $\gk = 1$.

\begin{thm}\label{thm-gk1MMP}
Let $X$ be a non-algebraic compact Kähler threefold with $\gk(X) = 1$. Let $X_\mmin$ be a minimal model of $X$ and $X_\mmin \to B$ the canonical fibration of $X_\mmin$. Then $X_\mmin \to B$ satisfies one of the following descriptions:
\begin{enumerate}[i)]
\item If a general fiber $F$ of $X_\mmin \to B$ is algebraic, then $F$ is either an abelian surface or a bielliptic surface;
\item If $F$ is not algebraic, then $F$ is either a K3 surface or a 2-torus, and there exists a finite Galois cover $\wt{B} \to B$ of $B$ and a smooth  fibration $\wt{X} \to \wt{B}$ whose fibers are all isomorphic to $F$, such that $\wt{X}$ is bimeromorphic to $X_{\mmin} \times_{B} \wt{B}$ over $\wt{B}$. Moreover, the monodromy action of $\pi_1(\wt{B})$ on $F$ preserves the holomorphic symplectic form. Finally if either $F$ is a K3 surface or $X_\mmin $ contains a curve which dominates $B$, then there exists a finite Galois base change as above such that $\wt{X} \to \wt{B}$ is isomorphic to the standard projection $F \times \wt{B} \to \wt{B}$. 
\end{enumerate}
\end{thm}

\begin{proof}

Since $X_\mmin$ has only isolated singularities, a general fiber $F$ of $X_\mmin \to B$ is a connected smooth surface. As  $K_F$ is torsion, the classification of surfaces shows that $F$ is either a K3 surface, an Enriques surface, a 2-torus, or a bielliptic surface. Since $X$, and thus $X_\mmin$ is non-algebraic, if $F$ is algebraic then by Fujiki's result~\cite[Proposition 7]{FujikiAlbC} $F$ is irregular, so $F$ can only be an abelian surface or a bielliptic surface, which proves $i)$. 

Assume that $F$ is not algebraic, then $F$ is either a K3 surface or a 2-torus and by~\cite{Campanaisotriv}, the fibration $X_\mmin \to B$ is isotrivial. By~\cite[Lemma 4.2]{HYLbimkod1}, there exists some finite map $\wt{B} \to B$ of $B$ and a smooth  fibration $\wt{X} \to \wt{B}$ all of whose fibers are isomorphic to $F$, such that $\wt{X}$ is bimeromorphic to $X_{\mmin} \times_{B} \wt{B}$ over $\wt{B}$. Up to taking the Galois closure of $\wt{B} \to B$, we can assume that $\wt{B} \to B$ is Galois.

Since $\tilde{f}$ is smooth and isotrivial, the fundamental group $\pi_1(\wt{B})$ acts on $F$ by monodromy transformations. Since $\wt{X}$ is assumed to be non-algebraic, we have $H^0(X,\gO_X^2) \ne 0$. Hence by the global cycle invariant theorem,  the $\pi_1(\wt{B})$-action on $F$ is symplectic.

As $\wt{X}$ is Kähler, again by the global cycle invariant theorem there exists a Kähler class on $F$ fixed by the induced monodromy action on $H^2(F,\bR)$.  It follows that the map $\pi_1(\wt{B}) \to \Aut(F)/\Aut_{0}(F)$ has finite image where $\Aut_0(F)$ denotes the identity component of $\Aut(F)$~\cite[Proposition 2.2]{LiebermanCompChowsch}.

In the case where $F$ is a K3 surface, $\Aut_0(F)$ is trivial, so  $\pi_1(\wt{B})$ acts as a finite group on $F$. Accordingly after some finite base change of $\tilde{f} : \wt{X} \to \wt{B}$, the fibration $\tilde{f}$ becomes a trivial.
Now assume that $F$  is a  2-torus and that $X_\mmin$ contains a curve dominating $B$. After another finite base change of  $\tilde{f} : \wt{X} \to \wt{B}$ we can assume that $\tilde{f}$ has a section $\gs : \wt{B} \to \wt{X}$, namely $\tilde{f}$ is a Jacobian fibration. Recall that $\pi_1(\wt{B}) \to \Aut(F)/\Aut_{0}(F)$ has finite image, so after a further finite base change  of  $\tilde{f} : \wt{X} \to \wt{B}$, we can assume that the monodromy action of $\pi_1( \wt{B})$ on $H^1(F,\bZ)$ is trivial. As $\tilde{f} : \wt{X} \to \wt{B}$ is a Jacobian fibration, we conclude that $\wt{X} \simeq F \times \wt{B}$ and that $\tilde{f}$ is isomorphic to the projection $F \times \wt{B} \to \wt{B}$.
 
 As before, both in the case where $F$ is a K3 surface or a 2-torus, up to taking the Galois closure of $\wt{B} \to B$ we can assume that $\wt{B} \to B$ is Galois.
\end{proof}

\section{Equivariant algebraic approximations of pairs}\label{sec-casparcas}

In this section, we will prove for some compact Kähler threefolds $X$ endowed with a $G$-action that for every $G$-invariant curve $C \subset X$, there exists a $G$-equivariant $C$-locally trivial algebraic approximation of the pair $(X,C)$. Results in Section~\ref{sec-bim} show that the quotients $X/G$ of these varieties cover all compact Kähler threefolds of Kodaira dimension $0$ or $1$ up to bimeromorphic transformations, hence will allow us to conclude the proof of Theorem~\ref{thm-mainpair} in Section~\ref{sec-concl}.

Before dealing with threefolds, we start by proving analogue statements concerning the existence of a $G$-equivariant $C$-locally trivial algebraic approximation for fibrations admitting a strongly locally trivial algebraic approximation and for surfaces in the next two subsections.

\ssec{Fibrations admitting a strongly locally trivial algebraic approximation}

\begin{lem}\label{lem-sltaa}
Let $X$ be a non-algebraic compact Kähler variety and $f : X \to B$ a surjective map onto a curve with algebraic fibers. Suppose that $X$ has a strongly locally trivial algebraic approximation $\pi : \cX \to \gD$ with respect to $f$, then for any subvariety $C \subset X$, up to shrinking $\gD$ the deformation $\pi$ induces a $C$-locally trivial algebraic approximation of $(X,C)$.

If moreover there exists a finite group $G$ acting $f$-equivariantly on $X$ and on $B$ and the algebraic approximation of $X$ in the assumption above is $G$-equivariant, then the induced $C$-locally trivial algebraic approximation is also $G$-equivariant for every $G$-invariant subvariety $C$.
\end{lem}

\begin{proof}
Since $X$ is non-algebraic and since the base and the fibers of $f$ are algebraic, by Campana's criterion~\cite[Corollaire in p.212]{CampanaCored} every subvariety of $X$ (in particular $C$) is contained in a finite number of fibers of $f$. We can thus apply Lemma~\ref{lem-strloctrivG} to conclude.
\end{proof}

\begin{cor}\label{cor-abaa}
Let $X$ be a non-algebraic compact Kähler variety and $f : X \to B$ a surjective map onto a curve. Let $G$ be a finite group acting $f$-equivariantly on $X$ and on $B$. Assume that a general fiber of $f$ is  an abelian variety, then for every $G$-invariant subvariety $C \subset X$, the pair $(X,C)$ has a $G$-equivariant $C$-locally trivial algebraic approximation.
\end{cor}

\begin{proof}
By~\cite[Theorem 1.6]{HYLbimkod1}, the fibration $f$ has a $G$-equivariant strongly locally trivial algebraic approximation. Hence Corollary~\ref{cor-abaa} follows from Lemma~\ref{lem-sltaa}.
\end{proof}

\ssec{Surfaces with a finite group action}\label{ssec-Gsurf}

\hfill

First we recall some Hodge-theoretical criteria for the existence of an algebraic approximation.

\begin{thm}[{Green's criterion~\cite[Proposition 1]{Buch2}}]\label{thm-Grenndensecrit}
Let $\pi : \cX \to B$ be a family of compact Kähler manifolds over a smooth base. If a fiber $X = \pi^{-1}(b)$ satisfies the property that the composition of the Kodaira-Spencer map and the contraction with some Kähler class $[\go] \in H^1(X,\gO_X^1)$
$$
\begin{tikzcd}[cramped, row sep = 0, column sep = 40]
\mu_{[\go]} :  T_{B,b} \arrow[r, "\mathrm{KS}"] & H^1(X, T_{X})  \arrow[r, "\cupp {[\go]}"] &  H^2(X, T_{X} \otimes \gO_X)  \arrow[r] &  H^2(X,\cO_X) \\ 
\end{tikzcd}
$$
is surjective, then there exists a sequence of points in $B$ parameterizing algebraic members which converges to $b$.
\end{thm}

The following is a variant of Theorem~\ref{thm-Grenndensecrit} when the variety $X$ is endowed with a finite group action.

\begin{thm}[{\cite[Theorem 9.1]{GrafDefKod0}}]\label{thm-Gdensecrit}
Let $X$ be a compact Kähler manifold with an action of a finite group $G$. Suppose that the universal deformation space of $X$ is smooth. If there exists a $G$-invariant Kähler class $[\go] \in H^1(X,\gO_X^1)$ such that the following composition of maps
$$
\begin{tikzcd}[cramped, row sep = 0, column sep = 40]
\mu_{[\go]} : H^1(X, T_{X})  \arrow[r, "\cupp {[\go]}"] &  H^2(X, T_{X} \otimes \gO_X)  \arrow[r] &  H^2(X,\cO_X) \\ 
\end{tikzcd}
$$
is surjective, then $X$ has a $G$-equivariant algebraic approximation.
\end{thm}

The following is an easy application of Theorem~\ref{thm-Gdensecrit}.

\begin{lem}\label{lem-GaaKtrivsurf}
Let $S$ be a non-algebraic compact Kähler surface and $G$ a finite group acting on $S$. If $K_S \simeq \cO_S$, namely if $S$ is either a K3 surface or a 2-torus, then $S$ has a $G$-equivariant algebraic approximation.
\end{lem}

\begin{proof}
Since $S$ is a surface with trivial $K_S$, the universal deformation space of $S$ is smooth. Also, we have the isomorphism $T_S \simeq \gO_S^1$ defined by the contraction with a fixed holomorphic symplectic form. So for a $G$-invariant Kähler class $[\go]$, the map $\mu_{[\go]}$ defined in Theorem~\ref{thm-Gdensecrit} with $\cX \to B$ replaced by the family of K3 surfaces $\cS_U \to U$ has the factorization
\begin{equation}\label{fact-mu}
\begin{tikzcd}[cramped, row sep = 0, column sep = 40]
\mu_{[\go]} : H^1(S, T_{S}) \simeq H^1(S, \gO_S^1) \arrow[r, "\cupp {[\go]}"] & H^2(S, \gO_S^2)  \simeq H^2(X,\cO_X). \\ 
\end{tikzcd}
\end{equation}
Since $[\go]^2 \ne 0$, the map $\mu_{[\go]}$ is non-zero. Moreover since $h^2(S,\cO_S) = 1$, the map $\mu_{[\go]}$ has to be surjective. Hence Lemma~\ref{lem-GaaKtrivsurf} is a consequence of Theorem~\ref{thm-Gdensecrit}.
\end{proof}

Lemma~\ref{lem-2toreaa} and~\ref{lem-K3aa} concern $C$-locally trivial algebraic approximations of a pair $(S,C)$ for $K$-trivial surfaces.

\begin{lem}\label{lem-2toreaa}
Let $S$ be a non-algebraic 2-torus and let $G$ be a finite group acting on $S$. Let $C \subset S$ be a $G$-invariant curve. Then the pair $(S,C)$ has a $G$-equivariant $C$-locally trivial algebraic approximation.
\end{lem}

\begin{proof}
Since $S$ is a non-algebraic 2-torus containing a curve, it is a smooth isotrivial elliptic fibration $f : S \to B$ and the only curves of $S$ are fibers of $f$. As the $G$-action sends curves to curves, the fibration $f$ is $G$-equivariant.  We thus conclude by Corollary~\ref{cor-abaa} that $(S,C)$ admits a $G$-equivariant $C$-trivial algebraic approximation.
\end{proof}

\begin{lem}\label{lem-K3aa}
Let $S$ be a non-algebraic K3 surface and let $G$ be a finite group acting on $S$. Let $C \subset S$ be a $G$-invariant curve. Then $(S,C)$ has a $G$-equivariant $C$-locally trivial algebraic approximation. When the algebraic dimension $a(S)$ of $S$ is zero, more precisely the deformation $\cS \to \gD$ of $S$ over the Noether-Lefschetz locus preserving the classes of each irreducible component of $C$ in the universal deformation of $S$ preserving the $G$-action is  a $G$-equivariant $C$-locally trivial algebraic approximation.
\end{lem}

\begin{proof}

First we note that since $H^2(S,\bZ)^G$ is a sub-$\bZ$-Hodge structure of $H^2(S,\bZ)$ of weight 2, if the $G$-action does not preserve the holomorphic symplectic form, then $H^2(S,\bZ)^G$ is concentrated in bi-degree $(1,1)$. As the intersection of $H^{1,1}(S)^G$ with the Kähler cone $\mathcal{K}_S \subset H^2(S,\bC)$ is not $0$, we deduce that $H^2(S,\bZ)^G$ contains a Kähler class, which is in contradiction with the hypothesis that $S$ is non-algebraic. We deduce that the $G$-action preserves the holomorphic symplectic form of $S$.

Since $S$ is assumed to be non-algebraic, according to whether $a(S) = 0$ or $1$ only two situations can happen:
\begin{enumerate}
\item  $a(S) = 0$: every curve in $S$ is a disjoint union of trees of smooth $(-2)$-curves intersecting transversally;
\item  $a(S) = 1$:  $S$ is an elliptic fibration $f : S \to B$ and the $G$-action sends fibers to fibers.
\end{enumerate}
In the second situation, we can apply Corollary~\ref{cor-abaa} to get a $G$-equivariant $C$-locally trivial algebraic approximation of $(S,C)$ as we did in the proof of Lemma~\ref{lem-2toreaa}.

In the first situation, let us write $C = \cup_{i \in I} C_i$ where the $C_i$'s are irreducible components of $C$. 
Since the universal deformation space of $S$ is smooth, its locus preserving the $G$-action can be identified with an open subset of $H^1(S,T_S)^G$. As the group action $G$ on $S$ is symplectic, the isomorphism $T_S \simeq \gO_S^1$ defined by the contraction with a fixed holomorphic symplectic form induces an isomorphism 
$$H^1(S,T_S)^G \simeq H^1(S,\gO_S^1)^G.$$
Under this identification, the universal deformation space $\gD$ of $S$ preserving the $G$-action and the curve classes $[C_i]$ can be identified with an open subset $U$ of 
$$V \colonec H^1(S,\gO_S^1)^G\cap \<[C_i] \>^\perp_{i\in I}$$
where $ \<[C_i] \>_{i\in I}$ denotes the linear subspace of $H^1(S,\gO_S^1)$ spanned by the classes $[C_i]$ and the orthogonality is defined with respect to the cup product. Since $ \<[C_i] \>_{i\in I}$ is $G$-invariant  and since the $G$-action preserves the cup product, the orthogonal $ \<[C_i] \>^\perp_{i\in I}$ is also $G$-invariant. Therefore $V =  \<[C_i] \>^\perp_{i\in I}$.

Since $S$ is not algebraic, the curve classes $[C_i]$ cannot generate the whole $H^1(S,\gO_S^1)$, hence $V \ne 0$ and let $v$ be a non-zero element in $V$. As $C_i^2 < 0$ for all $i$, by the Hodge index theorem  $v^2 >0$. If $[\go]$ is a Kähler class, then again by the Hodge index theorem we have $v \cdot  [\go] \ne 0$. Using the factorization~\eqref{fact-mu}, we see again that since  $h^2(S,\cO_S) = 1$, the map $\mu_{[\go]}$ defined in Theorem~\ref{thm-Grenndensecrit} with $\cX \to B$ replaced by the $G$-equivariant deformation $\cS \to \gD$ of $S$ over the Noether-Lefschetz locus $\gD$, is surjective. Therefore by Theorem~\ref{thm-Grenndensecrit}, $\cS \to \gD$ is an algebraic approximation of $S$. Since the curve classes $[C_i] \in H^2(S,\bC)$ remains of type $(1,1)$, $\cS \to \gD$ induces 
  for each $i$, a deformation $(\cS,\cC_i)$ of the pair $(S,C_i)$. It remains to show that $(\cS,\cC \colonec \cup_{i \in I} \cC_i) \to \Delta$ is a $C$-locally trivial deformation. 

Let us decompose $C = \sqcup_{i=1}^m C'_i$ into its connected components. As we mentioned before,  each $C'_i$ is a tree of smooth $(-2)$-curves intersecting transversally. Therefore up to shrinking $\gD$, if $\cC = \sqcup_{i=1}^m \cC'_i$ denotes the decomposition of $\cC$ into its connected components, then up to reordering the indices $i$, each fiber of $\cC'_i \to \gD$ is still a tree of $(-2)$-curves isomorphic to $C'_i$.

Since a tree of smooth $(-2)$-curve on a surface can be contracted to a rational double point, there exists a bimeromorphic morphism $\nu : \cS \to \cS'$ over $\gD$ such that for each fiber $\cS_t$ of $\cS \to \gD$, the restriction of $\nu$ to $\cS_t$ is the contraction of $\cC'_i \cap \cS_t$ to a rational double point~\cite[Theorem 2]{RiemenschneiderDefoRatSing}. Since fibers of $\cC'_i \to \gD$ are all isomorphic, the singularity type of $\nu(\cC'_i \cap \cS_t) \subset \cS_t$ does not depend on $t \in \gD$. As the germs of a rational double point of a fixed type on a surface are all isomorphic, up to shrinking $\gD$ there exists a neighborhood $\cU_i \subset \cS'$ of $\nu(\cC'_i)$ such that the pair $\(\cU_i, \nu(\cC'_i)\)$ is isomorphic over $\gD$ to the trivial product $\(U_i \times \gD, \nu(C'_i)\)$ with $U_i \colonec \cU_i \cap \nu(S)$. It follows that $\(\cS',\nu(\cC)\) \to \gD$ is a $\nu(C)$-locally trivial deformation of $\(\nu(S), \nu(C)\)$, hence $(\cS, \cC) \to \gD$ is $C$-locally trivial by Lemma~\ref{lem-defcontrpoint}.
\end{proof}

For the sake of completeness, we conclude the present subsection by the following proposition which will not be used latter in the article. It is the generalization of~\cite[Lemma 5.1]{Schrackdefo} in the $G$-equivariant setting.

\begin{pro}
Let $S$ be a compact Kähler surface and $G$ a finite group acting on $S$. Whenever $C \subset S$ is a curve or empty, the pair $(S,C)$ has a $G$-equivariant $C$-locally trivial algebraic approximation.
\end{pro}

\begin{proof}
We may assume that $S$ is non-algebraic. If the algebraic dimension $a(S)$ of $S$ is $1$, then $S$ is an elliptic fibration and we can use Corollary~\ref{cor-abaa} to conclude. If $a(S) = 0$, then the minimal model $S'$ of $S$ is either a 2-torus or a K3 surface and the map $\nu : S \to S'$ is $G$-equivariant. By Lemma~\ref{lem-GaaKtrivsurf},~\ref{lem-2toreaa}, and~\ref{lem-K3aa}, the pair $(S', \nu(C))$ has a $G$-equivariant $\nu(C)$-locally trivial algebraic approximation. Hence by  Lemma~\ref{lem-defcontrpoint}, $(S, C)$ has a $G$-equivariant $C$-locally trivial algebraic approximation.
\end{proof}

\ssec{K3 fibrations}

\begin{lem}\label{lem-K3fibaa}
Let $X \colonec S \times B$ where $S$ is a non-algebraic K3 surface and $B$ is a smooth projective curve. Let $G$ be a finite group acting  on $B$ and  on $S$ and let $G$ act on $X$ by the product action. Whenever $C \subset X$ is a $G$-invariant curve or empty, the pair $(X,C)$ has a $G$-equivariant $C$-locally trivial algebraic approximation. 
\end{lem}

\begin{proof}
Let $p_1 : S \times B \to S$ denote the first projection. As the $G$-action on $S \times B$ is a product action, the image $C' \colonec p_1(C)$ is a $G$-invariant curve. By Lemma~\ref{lem-GaaKtrivsurf} and~\ref{lem-K3aa}, there exists a $G$-equivariant $C'$-locally trivial algebraic approximation $\pi: (\cS,\cC') \to \gD$ of the pair $(S,C')$. Let $\cU \subset \cS$ be a neighborhood of $\cC'$ such that there exists an isomorphism $\cU \simeq U \times \gD$ over $\gD$, so $U \times B \times \gD \simeq \cU \times B$ over $\gD$. Since $U \times B$ is a neighborhood of $C$ and since $C$ is $G$-invariant, Lemma~\ref{lem-Gloctriv} implies that the algebraic approximation $\Pi : \cX \colonec \cS \times B \to \gD$ of $X$ defined by the composition of $\pi$ with the first projection $\cS \times B \to \cS$ induces a $C$-locally trivial algebraic approximation of $(X,C)$.
\end{proof}

\ssec{2-torus fibrations}

\hfill

Before we study the existence of ($C$-)locally trivial algebraic approximations of a pair $(X,C)$ in the case of 2-torus fibrations, let us first prove a statement concerning the existence of multisections of a torus fibration \emph{via} strongly locally trivial perturbation.

\begin{lem}\label{lem-multsec}
Let $f : X \to B$ be a smooth torus fibration whose total space $X$ is compact Kähler. There exists an arbitrarily small strongly locally trivial deformation $f' : X' \to B$ of $f$ such that $f'$ has a multisection.
Moreover if $f$ is endowed with an $f$-equivariant $G$-action for some finite group $G$, then one can choose the above deformation to be $G$-equivariant.
\end{lem}

\begin{proof}
The construction of an arbitrarily small deformation of $f$ possessing a multi-section already appeared in~\cite{ClaudonToridefequiv}. We will recall how this deformation is constructed and prove that it is strongly locally trivial along the way.

Let $J \to B$ be the Jacobian fibration associated to $f$ and $\cJ$ its sheaf of sections. The sheaf can be defined by the exact sequence
$$
\begin{tikzcd}[cramped, row sep = 5, column sep = 40]
0 \ar[r] & \bH_\bZ  \arrow[r]  & \cE \ar[r] & \cJ \ar[r] & 0 \\
\end{tikzcd}
$$
where $\bH_\bZ \colonec R^{2g-1}f_*\bZ$ and $\cE \colonec \cH / \cH^{g,g-1} \colonec R^{2g-1}f_*\bC / R^{g-1}f_*\gO^g_{X/B}$.  To each isomorphism class of $J$-torsor $g : Y \to B$, one can associate in a biunivocal way, an element $\eta(g) \in  H^1(B,\cJ)$ satisfying the property that $g$ has a multisection if and only if $\eta(g)$ is torsion (\emph{cf.}~\cite[Section 2.2]{ClaudonToridefequiv}). Moreover, if  $\exp : V \colonec H^1(B ,\cE) \to H^1(B, \cJ)$ denotes the morphism induced by the quotient $\cE \to \cJ$, then there exists a family
\begin{equation}\label{fam-Jtors}
\begin{tikzcd}[cramped, row sep = 20, column sep = 40]
\cX  \arrow[r, "q"] \arrow[d, "\pi", swap] & V \times B  \ar{dl}{\pr_1} \\
V &  \\
\end{tikzcd}
\end{equation}
of $J$-torsor such that for each $v \in V$, the element in $V$ associated to the $J$-torsor $\pi^{-1}(v) \to B$ is $\eta(f) + \exp(v)$. 

Concretely, the above family is constructed as follows. The map
\begin{equation*}
\begin{split}
V & \to  H^1(B, \cJ) \\
v & \mapsto  \eta(f) + \exp(v), 
\end{split}
\end{equation*}
defines an element
$${\eta}^V \in \Map\(V, H^1(B,\cJ)\) \simeq H^0(V,\cO_V) \otimes H^1(B,\cJ) \simeq H^1(V \times B, \pr_2^*\cJ)$$
where $\Map\(V, H^1(B,\cJ)\)$ denotes the space of holomorphic maps between $V$ and $H^1(B,\cJ)$. So one can find a covering $\cup_{i=1}^nU_i = B$ of $B$ by open subsets such that ${\eta}^V$ represents a \v{C}ech 1-cocycle
$$\eta^V_{ij} \in \Gamma(V \times U_{ij}, \pr_2^*\cJ) \simeq \Map\(V, \Gamma(U_{ij}, \cJ)\)$$
where $U_{ij} \colonec U_i \cap U_j$. Let us write $X_i \colonec f^{-1}(U_i)$ and  $X_{ij} \colonec f^{-1}(U_{ij})$ for all $i$ and $j$. The 1-cocycle $(\eta^V_{ij})_{i,j}$ defines the transition maps $V \times X_{ij}  \to V \times X_{ij}$ which are translations by $\eta^V_{ij}$ and  the family $\cX \to V \times B$ is obtained by glueing $(V \times X_i  \to V \times U_i )_i$ together using these  transition maps. Since $q^{-1}(V \times U_i) \simeq V \times X_i $ over $V$ for all $i$, the family $\pi : \cX \to V$ is strongly locally trivial.

If $f : X \to B$ is endowed with an $f$-equivariant $G$-action for some finite group $G$, then this $G$-action induces an action on $\cE$ and on $\cJ$. The restriction to the $G$-invariant subspace $V^G \subset V$ of~\eqref{fam-Jtors} is a deformation of the $J$-torsor $f : X \to B$ preserving the equivariant $G$-action~\cite[Proposition 2.10]{ClaudonToridefequiv}. The proof that $V^G$ contains a dense subset of points parameterizing $J$-torsors having a multi-section is contained in the proof of~\cite[Theorem 1.1]{ClaudonToridefequiv}, which we sketch now and provide necessary references for the detail.

By Deligne's theorem, $W \colonec H^1(B, \bH_\bZ)$ is a pure Hodge structure of degree $2g$ and concentrated in bi-degrees $(g-1,g+1)$, $(g,g)$, and $(g+1,g-1)$~\cite[Section 2]{ZuckerHdgL2}. Let $W_K \colonec W \otimes K$ for any field $K$. If $F^\bullet W_\bC$ denotes the Hodge filtration, then $V$ is isomorphic to $W_\bC/F^gW_\bC$~\cite[Section 2]{ZuckerHdgL2}.
Let $\mu : W_\bR \to V$ denote the composition 
$$\mu : W_\bR \hto W_\bC  \to V.$$
Using the Hodge theory we see easily that $\mu$ is surjective, so $\mu(W_{\bQ})$ is dense in $V$. Since $G$ is finite, we have 
$$\mu(W_{\bQ}^G) \otimes \bR = \mu(W_{\bQ})^G \otimes \bR = V^G.$$
Therefore $\mu(W_{\bQ}^G) $ is dense in $V^G$.

Using the assumption that $X$ is Kähler, one can prove that the image of the $G$-equivariant class $\eta_G(f) \in H^1_G(B,\cJ)$ associated to $X$ (which is a refinement of $\eta(f)$, \emph{cf.}~\cite[Section 2.4]{ClaudonToridefequiv}) under the connection morphism
$$H^1_G(B,\cJ) \to H^2_G(B,\bH_\bZ)$$
is torsion~\cite[Proposition 2.11]{ClaudonToridefequiv}.
It follows that there exists $m \in \bZ_{>0}$ and $v_0 \in V^G$ such that $m\eta(f) = \exp(v_0)$. Therefore $\eta(f) + \exp\(v - \frac{v_0}{m}\)$ is torsion for each $v \in \mu(W_{\bQ}^G)$, so each of the fibrations $\cX_v \to B$ parameterized by the subset 
$$\mu(W_{\bQ}^G) - \frac{v_0}{m} \subset V^G$$
in the family~\eqref{fam-Jtors} has a multisection. As we saw that $\mu(W_{\bQ}^G) \subset V^G$ is dense, we conclude that the restriction of~\eqref{fam-Jtors} to $V^G$ is a deformation of $f: X \to B$ containing a dense subset of members having a multisection. 
\end{proof}

\begin{lem}\label{lem-2torfibaa}
Let $f: X \to B$ be a smooth isotrivial 2-torus fibration over a smooth projective curve $B$. Let $G$ be a finite group acting $f$-equivariantly on $X$ and on $B$ such that $X \to B$ coincides with the base change of $X/G \to B/G$ by $B \to B/G$. Whenever $C \subset X$ is a $G$-invariant  curve or empty, the pair $(X/G,C/G)$ has a $C/G$-locally trivial algebraic approximation. 
\end{lem}

\begin{proof}

First we assume that $f$ does not have any multisection. In particular, the curve $C$ is contained in a finite union of fibers of $f$. Using Lemma~\ref{lem-multsec} there exists an arbitrarily small strongly locally trivial, so in particular $C$-locally trivial, deformation of $f$ to some fibration which has a multisection. Thus up to replacing $f$ by this arbitrarily small deformation, we can assume that $f$ has a multisection.

Since $f$ has a multisection, there exists a finite base change $\tilde{f} : \wt{X} \to \wt{B}$ of $X \to B$ such that $\wt{X} \simeq S \times \wt{B}$ where $S$ is a fiber of $f$ and that $\tilde{f}$ is the second projection. After base changing with the Galois closure of $\wt{B} \to B/G$, we can assume that $\wt{B} \to B/G$ is Galois whose Galois group is denoted by $\wt{G}$ acting on $\wt{B}$ and on $S$ by monodromy transformations.

Let $\wt{C}$ be the pre-image of $C$ under the map $\wt{X} \to X$, which is $\wt{G}$-invariant by assumption. By Lemma~\ref{lem-Gloctriv}, it suffices to show that the pair $(\wt{X},\wt{C})$ has a $\wt{G}$-equivariant $\wt{C}$-locally trivial algebraic approximation.  Since $S \times \wt{B} \to \wt{B}$ is isomorphic to the base change of $\wt{X}/G \to \wt{B}/G$ by $\wt{B} \to B/G$, the $\wt{G}$-action on $S \times \wt{B}$ induces a $\wt{G}$-action on $S$ such that the first projection $p_1 : S \times \wt{B} \to S$ is $\wt{G}$-equivariant. As $C$ is $\wt{G}$-invariant, the curve $C' \colonec p_1(C)$ is also $\wt{G}$-invariant. By Lemma~\ref{lem-GaaKtrivsurf} and~\ref{lem-2toreaa}, there exists a $\wt{G}$-equivariant $C'$-locally trivial algebraic approximation $(\cS,\cC') \to \gD$ of the pair $(S,C')$. By repeating the same argument as in the proof of Lemma~\ref{lem-K3fibaa}, we conclude that the deformation 
$\Pi : \cS \times \wt{B} \to \gD$ induces a $\wt{G}$-equivariant $\wt{C}$-locally trivial algebraic approximation of the pair $(\wt{X},\wt{C})$.
\end{proof}

\ssec{Non-algebraic 3-tori}

\begin{lem}\label{lem-3toriaa}
Let $X$ be a non-algebraic 3-torus and $G$ a finite group acting on $X$. Then for every $G$-invariant curve $C \subset X$, the pair $(X,C)$ has a $G$-equivariant algebraic approximation. 
\end{lem}

\begin{proof}

First we assume that there exists a generically injective morphism $\nu : C' \to X$ from a smooth curve of geometric genus $\ge 2$ to $X$. Since $\nu$ factorizes through ${C'} \to J({C'}) \xto{j} X$ where $ J({C'})$ denotes the Jacobian associated to ${C'}$, the 3-torus $X$ contains an abelian variety of dimension $\ge 2$ which is  $j(J({C'})) \subset X$. As $X$ is non-algebraic, we have $\dim j(J({C}')) = 2$ and hence $X$ is a smooth isotrivial fibration $f : X \to B$ in abelian surfaces. As $X$ is assumed to be non-algebraic, the $G$-action on $X$ preserves the fibers of $f$. Hence we can apply Corollary~\ref{cor-abaa} to conclude that $(X,C)$ has a $G$-equivariant algebraic approximation.

Now assume that $X$ does not contain any curve of geometric genus $\ge 2$, then $C$ is a union of smooth elliptic curves. It follows that $X$ is a smooth isotrivial elliptic fibration $f: X \to S$. Moreover, the fibration $f$ does not have any proper curve other than the fibers of $f$. Indeed, if such a curve $C'$ exists, then for any fiber $F$ of $f$ the image of $\ga : C' \times F \to X$ defined by $\ga(x,y) \colonec x + y$ is an algebraic surface, so necessarily contains a curve of geometric genus $\ge 2$ which is in contradiction with our assumption. 

Since the only curves of $X$ are fibers of $f$, the curve $C$ is a union of fibers of $f$. It also follows that the $G$-action preserves the fibers of $f$, so induces a $G$-action on $S$. By Lemma~\ref{lem-multsec}, there exists an arbitrarily small $G$-equivariant strongly locally trivial, hence $C$-locally trivial, deformation $f' : (X',C) \to B$ of $f$ having a multisection. For such an $X'$, we already saw that $X'$ contains at least one curve of geometric genus $2$, so that $(X',C)$, and hence $(X,C)$, have a $G$-equivariant $C$-locally trivial algebraic approximation.
\end{proof}

\section{Algebraic approximations of compact Kähler threefolds}\label{sec-concl}

We can now conclude the proof of Theorem~\ref{thm-mainpair}.

\begin{proof}[Proof of Theorem~\ref{thm-mainpair}]

Let $X$ be a non-algebraic compact Kähler threefold and let $X'$ be a bimeromorphic model of $X$ for which we wish to prove that whenever $C \subset X'$ is a curve or empty, the pair $(X',C)$ has a locally trivial and $C$-locally trivial algebraic approximation.
If the choice of $X'$ is isomorphic to the quotient $\wt{X}/G$ of some smooth variety $\wt{X}$ by a finite group $G$, then first of all $\wt{X}/G$ is $\bQ$-factorial. To prove that $(\wt{X}/G,C)$ has a locally trivial and $C$-locally trivial algebraic approximation, it suffices by Lemma~\ref{lem-Gloctriv} to prove that the pair $(\wt{X},\wt{C})$ has a $G$-equivariant $\wt{C}$-locally trivial algebraic approximation $(\wt{\cX},\wt{\cC}) \to \gD$ where $\wt{C}$ is the pre-image of $C$ under the quotient map $ \wt{X} \to \wt{X}/G$.

If $\gk(X) = 0$, we choose $X'$ to be a minimal model of $X$. In particular, $X'$ is $\bQ$-factorial and has at worst terminal singularities. By Proposition~\ref{pro-gk0MMP}, the variety $X'$ is a quotient $\wt{X} / G$ by a finite group where $\wt{X}$ is either a non-algebraic 3-torus or the product of a non-algebraic K3 surface and an elliptic curve. If ${C} = \emptyset$, since $X'$ is minimal, the existence of a locally trivial algebraic approximation of $X'$ results from~\cite[Theorem 1.4]{GrafDefKod0}. If $C$ is a curve, the existence of a $G$-equivariant $\wt{C}$-locally trivial algebraic approximation of $(\wt{X},\wt{C})$ is a consequence of Lemma~\ref{lem-3toriaa}  or of Lemma~\ref{lem-K3fibaa} together with Lemma~\ref{lem-GactionK3E}, according to wether $\wt{X}$ is a 3-torus or the product of a non-algebraic K3 surface and an elliptic curve. 

If $\gk(X) = 1$, then by Theorem~\ref{thm-gk1MMP} there are two cases to be distinguished. If we are in the first case of Theorem~\ref{thm-gk1MMP}, with the same notation therein we take $X' = X_\mmin$, so in particular $X'$  is $\bQ$-factorial with at worst terminal singularities. Since the canonical fibration $X_\mmin \to B$ has a strongly locally trivial algebraic approximation~\cite[Theorem 1.6, Corollary 6.2]{HYLbimkod1}, we can apply Lemma~\ref{lem-sltaa} and deduce that $(X',C)$ has a locally trivial and $C$-locally trivial algebraic approximation for every curve $C \subset X'$. If we are in the second case of Theorem~\ref{thm-gk1MMP}, with the same notation therein we take $X' = \wt{X} /G$ where $G \colonec \Gal(\wt{B}/B)$. By~\cite[Proposition 4.7]{HYLbimkod1}, the variety $X'$ has at worst terminal singularities. The existence of a $G$-equivariant $\wt{C}$-locally trivial algebraic approximation of $(\wt{X},\wt{C})$ is a consequence of Lemma~\ref{lem-K3fibaa} or Lemma~\ref{lem-2torfibaa}, according to wether $\wt{X} \to \wt{B}$ is a fibration in K3 surfaces or 2-tori.
\end{proof}

As was mentioned in the introduction, the combination of Proposition~\ref{pro-red} and Theorem~\ref{thm-mainpair} proves Theorem~\ref{thm-main}, the existence of an algebraic approximation of any compact Kähler threefold of Kodaira dimension 0 or 1. 

Finally we prove Proposition~\ref{pro-gk2aa}, which concerns threefold of Kodaira dimension 2.

\begin{proof}[Proof of Proposition~\ref{pro-gk2aa}]

As an output of the Kähler MMP for threefolds and the abundance theorem (\emph{cf}. the beginning of Section~\ref{sec-bim}), a compact Kähler threefold $X$ with $\gk(X) = 2$ is bimeromorphic to an elliptic fibration $X' \to B'$ with $X'$ being normal and $B'$ a projective surface. Let
\begin{equation}\label{map-res}
\begin{tikzcd}[cramped, row sep = 0, column sep = 20]
X & Y'  \arrow[r, "\nu"] \ar[l , "\mu" swap] & {X'} \\
\end{tikzcd}
\end{equation}
be a resolution of the bimeromorphic map $X \dto X'$ where $\mu$ is bimeromorphic. Since $X'$ is normal, there exists a subvariety $C \subset X'$ of dimension at most $1$ such that the restriction of $\nu$ to $Y' \bss \nu^{-1}(C)$ is an isomorphism onto $X' \bss C$. Accordingly $f' : Y' \to B'$ is still an elliptic fibration. Let $D' \subset B'$ denote the locus parameterizing singular fibers of $f'$ and let $(B,D) \to (B',D')$ be a log-resolution of the pair $(B',D')$. Let $\nu' : Y \to Y' \times_{B'} B$ be a desingularization of $Y' \times_{B'} B$. As  $U \colonec Y' \times_{B'} (B \bss D) \to B \bss D$ and hence $U$ are smooth, we can assume that the restriction of $\nu'$ to the Zariski open $\nu'^{-1}(U)$ is an isomorphism onto $U $. It follows that $Y \to B$ is an elliptic fibration whose locus of singular fibers is contained in the normal crossing divisor $D$. 

Let $\eta : Y \to Y' \times_{B'} B \to Y' \to X$ denote the composition, which is bimeromorphic. Since both $Y$ and $X$ are smooth, we have $\eta_*\cO_Y = \cO_X$ and $R^1\eta_*\cO_Y = 0$. We can therefore apply~\cite[Theorem 2.1]{RanStabMap} as in the proof of Lemma~\ref{lem-deform} to conclude that if Question~\ref{Q-ellpaa} has a positive answer for the elliptic fibration $Y \to B$, then $X$ has an algebraic approximation by~\cite[Theorem 2.1]{RanStabMap}.

\end{proof}

\section*{Acknowledgement}
The author is supported by the SFB/TR 45 "Periods, Moduli Spaces and Arithmetic of Algebraic Varieties" of the DFG (German Research Foundation). He would like to thank F. Gounelas, C.-J. Lai, S. Schreieder, A. Soldatenkov, and C. Voisin for questions, remarks, and general discussions on various subjects related to this work.

\bibliographystyle{plain}
\bibliography{Kod3kappa01}

\end{document}